\documentclass[letterpaper, 10 pt, conference]{ieeeconf}

\IEEEoverridecommandlockouts	
\makeatletter
\let\NAT@parse\undefined
\makeatother

\usepackage{amsmath}
\usepackage{amssymb}
\usepackage{mathtools}
\usepackage{mathrsfs}
\usepackage{comment}
\usepackage{pgf}

\DeclareMathOperator*{\relu}{ReLU}		
\newcommand{\F}{\mathcal{F}}	
\newcommand{\G}{\mathcal{G}}	

\newcommand{\I}{\mathcal{I}}	
\newcommand{\J}{\mathcal{J}}	
\newcommand{\K}{\mathcal{K}}	

\newcommand{\N}{\mathbb{N}}	
\newcommand{\R}{\mathbb{R}}	
\newcommand{\Rpl}{\mathbb{R}_+}	
\newcommand{\Rbar}{\overline{\mathbb{R}}}	
\newcommand{\Rd}{\R^d}		
\newcommand{\Rn}{\R^n}		

\newcommand{\E}{\mathbb{E}}	
\renewcommand{\P}{\mathcal{P}}	
\newcommand{\D}{\mathcal{D}}	

\DeclareMathOperator*{\argmin}{arg\,min}	

\DeclareMathOperator*{\dom}{dom}		
\DeclareMathOperator*{\relint}{ri}		
\DeclareMathOperator*{\conv}{conv}		
\DeclareMathOperator*{\cone}{cone}		
\DeclareMathOperator*{\epi}{epi}		
\DeclareMathOperator*{\lip}{Lip}		
\newcommand{\prsp}[1]{\mathscr{P}_{#1}}		

\newcommand{\bx}{\mathbf x}			
\newcommand{\by}{\mathbf y}			
\newcommand{\pole}{\pi^{\star}}			
\newcommand{\poll}{\pi_{\theta}}		

\usepackage{amsthm}
\theoremstyle{plain}
\newtheorem{theorem}{Theorem}
\newtheorem{proposition}{Proposition}

\newtheorem{lemma}{Lemma}
\theoremstyle{definition}
\newtheorem{definition}{Definition}

\newtheorem{assumption}{Assumption}
\theoremstyle{remark}
\newtheorem{remark}{Remark}

\usepackage{hyperref}		
\hypersetup{%
	colorlinks=true,%
	linkcolor=black,%
	urlcolor=black,%
	citecolor=black,%
	filecolor=black%
}				
\usepackage[all]{hypcap}	

\title{\LARGE \bf
	Tight Certified Robustness via Min-Max Representations\\
	of ReLU Neural Networks
}

\author{%
	Brendon G.\ Anderson \and Samuel Pfrommer \and Somayeh Sojoudi%
	\thanks{The authors are with the University of California, Berkeley. Emails: {\tt \{\href{mailto:bganderson}{bganderson},\href{mailto:sam.pfrommer@berkeley.edu}{sam.pfrommer},\href{mailto:sojoudi@berkeley.edu}{sojoudi}\}@berkeley.edu}.}%
	\thanks{This work was supported by grants from ONR and NSF.}%
}

\begin{document}

\maketitle
\thispagestyle{empty}
\pagestyle{empty}

\begin{abstract}
	The reliable deployment of neural networks in control systems requires rigorous robustness guarantees. In this paper, we obtain \emph{tight} robustness certificates over convex attack sets for min-max representations of ReLU neural networks by developing a convex reformulation of the nonconvex certification problem. This is done by ``lifting'' the problem to an infinite-dimensional optimization over probability measures, leveraging recent results in distributionally robust optimization to solve for an optimal discrete distribution, and proving that solutions of the original nonconvex problem are generated by the discrete distribution under mild boundedness, nonredundancy, and Slater conditions. As a consequence, optimal (worst-case) attacks against the model may be solved for \emph{exactly}. This contrasts prior state-of-the-art that either requires expensive branch-and-bound schemes or loose relaxation techniques. Experiments on robust control and MNIST image classification examples highlight the benefits of our approach.
\end{abstract}

\section{Introduction}
\label{sec: intro}

Neural networks are rapidly being deployed in control systems as a means to efficiently model nonlinear systems \cite{chen1990non}, controllers \cite{ge2013stable}, and reinforcement learning policies \cite{levine2016end}. However, the performance of neural networks can be extremely sensitive to small fluctuations in their input data \cite{szegedy2014intriguing}. For example, \cite{eykholt2018robust,liu2019perceptual} show that image classification models can be fooled into misclassifying vehicle traffic signs when subject to digital or physical adversarial attacks, i.e., human-imperceptible data perturbations designed to cause failure. This unreliable behavior is directly at odds with the robustness guarantees required in safety-critical control settings such as autonomous driving \cite{bojarski2016end}.

In light of these sensitivities, much effort has been placed on developing methods to rigorously certify the robustness of neural networks, with a large emphasis on models using the popular ReLU activation function. However, certifying a neural network's robustness generally amounts to solving an intractable nonconvex optimization problem \cite{katz2017reluplex}. Three major lines of work have focused on overcoming this intractability: convex relaxations, Lipschitz bounding, and branch-and-bound methods (all discussed further in Section~\ref{sec: related_works}).

In this paper, we utilize an alternative representation of ReLU neural networks as a means to efficiently compute tight robustness certificates using convex optimization (and hence in polynomial time). As a consequence, we are able to exactly compute optimal (worst-case) attacks, which is generally not possible using the popular local search-based attack methods such as projected gradient descent \cite{madry2018towards} and the Carlini-Wagner attack \cite{carlini2017towards}.

\subsection{Related Works}
\label{sec: related_works}

\subsubsection{Robustness Certification}
Certifying the robustness of a model amounts to solving the nonconvex optimization $\inf_{x\in X} g(x)$, where $X$ is a set of possible inputs or attacks (i.e., the ``threat model''), and $g(x)$ is either the model output at an input $x$, or some linear transformation of the model output (e.g., a classifier's margin between two classes).

Convex relaxations work by optimizing over a convex outer-approximation of the set $g(X)$ of possible outputs. Popular relaxations involve linear bounding and programming \cite{wong2018provable,zhang2018efficient}, and semidefinite programming \cite{raghunathan2018semidefinite,fazlyab2020safety}, which constitutes a line of increasingly accurate yet computationally complex relaxations. Convex relaxation-based certificates remain loose in general, and their looseness has been shown to increase with model size \cite{anderson2023towards}.

The Lipschitz constant of a model provides a certified bound on how much the model output may change given some change in its input. Thus, bounds on the Lipschitz constant can yield efficient robustness certificates \cite{fazlyab2019efficient}. A number of works are devoted to computing Lipschitz bounds, but it has proven difficult to obtain tight enough bounds to grant meaningful certificates \cite{fazlyab2019efficient,weng2018towards,virmaux2018lipschitz,jordan2020exactly}.

Mixed-integer programming and branch-and-bound have also been applied to robustness certification for ReLU neural networks \cite{tjeng2019evaluating,anderson2020tightened,wang2021beta}. In contrast to convex relaxations and Lipschitz bounding, these methods are capable of obtaining tight certificates if they are run to convergence, but this incurs exponential computational complexity, preventing them from scaling to practically-sized models \cite{wang2021beta}. Some methods allow for early termination of their optimizations to yield more efficient, yet loose certificates \cite{wang2021beta}.

\subsubsection{Representations of ReLU Neural Networks}
ReLU neural networks are defined by compositions $g = \mathcal{A}_L\circ\sigma\circ\cdots\circ\sigma\circ\mathcal{A}_1$ with affine functions $\mathcal{A}_l$ and elementwise activation functions $\sigma = \relu \colon x \mapsto \max\{0,x\}$. The most prevalent alternative representation of such a model is as a piecewise linear function, i.e., a finite polyhedral partition of $\Rd$ with associated affine functions that agree with $g$ on each polyhedron \cite{montufar2014number,arora2018understanding}. Another representation is as a rational function when working with tropical algebra, where addition $\oplus$ and multiplication $\otimes$ are defined by $x\oplus y = \max\{x,y\}$ and $x\oplus y = x+y$ \cite{zhang2018tropical}. Finally, min-max representations---discussed in Section~\ref{sec: method}---have recently been introduced, where $g$ is expressed as the pointwise minimum of pointwise maxima of affine functions. These works restrict their focus to showcasing the impressive approximation capabilities of ReLU models and their alternative representations.

\subsection{Contributions}

\begin{enumerate}
	\item We show that ReLU neural networks admit min-max representations and hence such representations are universal function approximators.
	\item By lifting the certification to an infinite-dimensional problem over probability measures, we prove that, under mild boundedness, nonredundancy, and Slater conditions, \emph{exact} solutions to the original nonconvex problem are efficiently obtained for min-max representations via reduction to a tractable finite-dimensional convex optimization problem.
	\item Experiments on robust control and MNIST image classification examples demonstrate the effectiveness of our approach.
\end{enumerate}

To the best of our knowledge, our work is the first to grant \emph{tight} robustness certificates in polynomial time amongst those considering general ReLU neural networks and their alternative representations.\footnote{See \cite{awasthi2019robustness} for a polynomial time solution to the special case of $2$-layer ReLU models.}

\subsection{Organization}
In Section~\ref{sec: method}, we introduce and analyze the min-max representation of ReLU neural networks. We develop our tight robustness certificates in Section~\ref{sec: theory}. Experiments illustrating the effectiveness of our approach are given in Section~\ref{sec: exper}, and concluding remarks are made in Section~\ref{sec: conc}.

\subsection{Notations}
The sets of natural, real, nonnegative real, and extended real numbers are denoted by $\N$, $\R$, $\Rpl$, and $\Rbar = \R\cup\{-\infty,\infty\}$ respectively. Throughout, we let $\I,\J,\K\subseteq\N$ denote index sets $\{1,\dots,m\}$, $\{1,\dots,n\}$, and $\{1,\dots,p\}$, respectively. The cardinality, convex hull, conic hull, and relative interior of a subset $X$ of $\Rd$ are denoted by $|X|$, $\conv(X)$, $\cone(X)$, and $\relint(X)$, respectively. Furthermore, we define $\mathcal{B}(X)$ to be the Borel $\sigma$-algebra on $X$. We denote the set of probability measures on the measurable space $(X,\mathcal{B}(X))$ by $\P(X)$. For $x\in\Rd$, the Dirac measure centered at $x$ is denoted by $\delta_x$, which we recall is the probability measure defined by $\delta_x(A) = 0$ if $x \notin A$ and $\delta_x(A) = 1$ if $x\in A$ for all $A\in \mathcal{B}(X)$. The set of all Dirac measures with center in $X$ is defined to be $\D(X) = \{\mu \in \P(X) : \text{$\mu=\delta_x$ for some $x\in X$}\}$. The set of continuous functions from $\Rd$ into $\R$ is denoted by $C(\Rd,\R)$. The effective domain of a function $f\colon \Rd \to \Rbar$ is defined to be the set $\dom(f) = \{x\in\Rd : f(x) < \infty\}$. If $f$ is Borel measurable and $\mu$ is a probability measure on $(X,\mathcal{B}(X))$, then we denote the expected value of $f$ with respect to $\mu$ by $\E_{x\sim\mu} f(x) = \int_X f(x) d\mu(x)$. If $f$ is convex, the subdifferential of $f$ at $x$ is denoted by $\partial f(x)$. Throughout, we let $\|\cdot\|$ denote an arbitrary norm on $\Rd$, and we denote its dual norm by $\|\cdot\|_*$.

\section{Min-Max Affine Functions}
\label{sec: method}

In this section, we formally define min-max affine functions, discuss works related to these functions, and show that every ReLU neural network admits such a representation.

\begin{definition}
	\label{def: min-max_affine_function}
	A function $g\colon\Rd\to\R$ is a \emph{min-max affine function} if there exist $\I,\J_1,\dots,\J_{|\I|}\subseteq\N$ and associated $a_{ij}\in\Rd,~b_{ij}\in\R$ such that $g(x) = \min_{i\in\I}\max_{j\in\J_i}(a_{ij}^\top x + b_{ij})$ for all $x\in\Rd$. In this case, the function $x\mapsto \min_{i\in\I}\max_{j\in\J_i}(a_{ij}^\top x + b_{ij})$ is called the \emph{min-max representation} of $g$.
\end{definition}

The class of all min-max affine functions on $\Rd$ is denoted by $\G$. Notice that $g\in\G$ is the pointwise minimum of $m=|\I|$ convex functions $g_i \colon x\mapsto \max_{j\in\J_i}(a_{ij}^\top x + b_{ij})$, and is therefore nonconvex in general. Without loss of generality, we henceforth assume that, for every $g\in\G$, there exists some $\J\subseteq\N$ with $n=|\J|$ such that the min-max representation of $g$ satisfies $\J_i=\J$ for all $i\in\I$.\footnote{This is without loss of generality, since the value $g(x)$ does not change upon appending affine global underestimators of the convex function $g_i \colon x\mapsto \max_{j\in\J_i}(a_{ij}^\top x + b_{ij})$ to the set of affine components $x\mapsto a_{ij}^\top x + b_{ij}$ of $g$. In other words, $\min_{i\in\I}\max_{j\in\J_i}(a_{ij}^\top x + b_{ij}) = \min_{i\in\I}\max_{j\in\J}(a_{ij}^\top x + b_{ij})$ if one defines $\J = \{1,\dots,n\}$ with $n = \max_{i\in\I}|\J_i|$ and $a_{ij}=v_i,~b_{ij}=g_i(0)$ for $j\in\J\setminus\J_i$, for all $i\in\I$, where $v_i\in\Rd$ is a subgradient of $g_i$ at $0$ (which exists by \cite[Theorem~23.4]{rockafellar1970convex}).}

\textbf{Related Works on Min-Max Affine Functions.} In the mathematics literature, min-max affine functions are also termed lattice polynomials \cite{marichal2009weighted}. The work \cite{velasco2022morphoactivation} shows that piecewise linear activation functions can be written in min-max affine form, and that neural networks learned with such representations perform highly in image classification tasks. The works \cite{bagirov2005max,bagirov2005supervised} study the theoretical and algorithmic aspects of training min-max affine functions to separate data, and show that separating $\{x_1,\dots,x_p\}\subseteq\Rd$ from $\{y_1,\dots,y_q\}\subseteq\Rd$ requires no more that $pq$ affine components. The authors of \cite{rister2017piecewise} use min-max representations of neural networks to characterize the training optimization landscape. The conversion of ReLU neural networks into min-max affine form is characterized in \cite[Theorem~4.15]{chen2020learning}. An algorithm for nonlinear system identification using min-max affine functions is developed in \cite{wang2002nonlinear}. Finally, min-max affine functions have been used as consistent statistical estimators, termed ``Riesz estimators,'' in the mathematical economics literature \cite{aliprantis2007riesz}. To the best of our knowledge, our work is the first to exploit min-max representations for purposes of robustness certification.

We now proceed with analyzing the representation power min-max affine functions. Let $\F$ be the class of all ReLU neural networks on $\Rd$. The following theorem shows that every ReLU neural network can be represented as a min-max affine function, and therefore min-max affine functions are universal function approximators.

\begin{theorem}
	\label{thm: uat}
	For every $f\in\F$, there exist $\I,\J\subseteq\N$ and $(a_{ij},b_{ij})\in\Rd\times\R$ for $i\in\I,~j\in\J$ such that
	\begin{equation}
		f(x) = \min_{i\in\I}\max_{j\in\J} (a_{ij}^\top x + b_{ij}) ~ \text{for all $x\in\Rd$}.
		\label{eq: uat}
	\end{equation}
	Hence, the class $\G$ of min-max affine functions is dense in $C(\Rd,\R)$ with respect to the topology of uniform convergence on compact sets.
\end{theorem}

\begin{proof}
	Every $f\in\F$ is piecewise affine, i.e., there is a finite collection $\mathcal{Q}$ of closed subsets of $\Rd$ such that $\Rd = \bigcup_{Q\in\mathcal{Q}}Q$ and $f$ is affine on every $Q\in\mathcal{Q}$. Hence, by \cite[Theorem~4.1]{ovchinnikov2002max}, there exist $\I,\J\subseteq\N$ and $(a_{ij},b_{ij})\in\Rd\times\R$ for $i\in\I,~j\in\J$ such that \eqref{eq: uat} holds. Thus, since $\F\subseteq\G$ and $\F$ is dense in $C(\Rd,\R)$ with respect to the topology of uniform convergence on compact sets \cite[Theorem~3.1]{pinkus1999approximation}, it holds that $\G$ is dense in $C(\Rd,\R)$ in the same sense.
\end{proof}

\section{Theoretical Robustness Certificates}
\label{sec: theory}

In this section, we develop our theoretical robustness certificates. Consider a model $g\colon \Rd \to \Rbar$, which may, for example, represent the output of a scalar-valued controller or the confidence of a binary classifier $f\colon\Rd\to\{1,2\}$ defined by $f(x) = 1$ if $g(x) \ge 0$ and $f(x)=2$ if $g(x)<0$. We consider the asymmetric robustness setting introduced in \cite{pfrommer2023asymmetric}, where nonnegative outputs $g(x) \ge 0$ are ``sensitive'' and we seek to certify that no input within some convex uncertainty set $X\subseteq\Rd$ causes the output to leave the sensitive operating regime. This asymmetric setting accurately models realistic adversarial situations. For example, an adversary may seek some imperceptible attack $x\in X = \{x'\in\Rd : \|x'- \overline{x}\| \le \epsilon\}$ to cause a vehicle's image classifier to predict ``no pedestrian'' ($g(x)<0$) when the nominal image $\overline{x}$ has a pedestrian in view (the sensitive regime; $g(x) \ge 0$), but not the other way around. We leave as future work the extension to vector-valued models.

Formally, the certification problem we seek to solve in this work is written
\begin{equation*}
	p^\star \coloneqq \inf_{x\in X} g(x).
\end{equation*}
The model $g$ is robust if and only if $p^\star \ge 0$. On the other hand, if $x^\star$ solves $p^\star$, then $x^\star$ is an optimal (worst-case) attack in $X$, and it is successful if $p^\star < 0$.

The problem $p^\star$ is nonconvex due to the nonconvexity of $g$. When $g$ is a min-max affine function, a naive reformulation of $p^\star$ yields that
\begin{equation*}
	p^\star = \inf_{(x,i,t)\in X\times \I\times\R} \{t : \text{$a_{ij}^\top x + b_{ij} \le t$ for all $j\in\J$}\},
\end{equation*}
which removes the nonconvexity in $x$ but is inefficient to solve in general due to the integer variable $i$. Alternatively, one may attempt to directly reformulate the problem into a convex one by minimizing the convex envelope of $g$. Although the resulting problem coincides with our convex reformulation $\underline{c}$ (introduced in Section~\ref{sec: lifting_the_problem}) on the relative interior of the direct reformulation's feasible set, it is difficult to obtain regularity conditions under which the direct reformulation holds with respect to its entire feasible set.

We propose an alternative approach to solving $p^\star$ that consists of three steps: 1) lift the problem to an optimization over probability measures, 2) leverage results and regularity conditions in distributionally robust optimization to make a finite-dimensional reduction of the problem, and 3) reformulate and solve the finite-dimensional reduction.

\subsection{Lifting the Problem}
\label{sec: lifting_the_problem}
We lift the problem to an optimization over probability measures by noting that $g(x) = \int_X g(x') d\delta_x(x') = \E_{x'\sim \delta_x} g(x')$ whenever $x\in X$:
\begin{equation*}
	p^\star = \inf_{\delta_x \in \D(X)} \E_{x'\sim\delta_x} g(x').
\end{equation*}
With this reformulation, the optimization objective is linear in the variable $\delta_x$, but the feasible set $\D(X)$ is nonconvex, making the problem intractable as written. Therefore, we consider relaxing the problem to an optimization over all probability measures:
\begin{equation*}
	p' \coloneqq \inf_{\mu\in\P(X)} \E_{x'\sim \mu} g(x').
\end{equation*}
The problem $p'$ is convex, but infinite-dimensional. We start by showing that the relaxation is exact:

\begin{proposition}
	\label{prop: exact_relaxation}
	It holds that $p' = p^\star$.
\end{proposition}

\begin{proof}
	Since $\D(X)\subseteq\P(X)$, it holds that $p' \le p^\star$. Now, let $\mu\in\P(X)$. Then, since $p^\star \le g(x')$ for all $x'\in X$, it holds that
	\begin{equation*}
		p^\star = \int_X p^\star d\mu(x') \le \int_X g(x') d\mu(x') = \E_{x'\sim \mu} g(x').
	\end{equation*}
	Since $\mu\in\P(X)$ is arbitrary, we conclude that $p^\star \le \inf_{\mu\in\P(X)}\E_{x'\sim \mu} g(x') = p'$. Hence, $p' = p^\star$.
\end{proof}

Next, we show that solutions of the nonconvex problem $p^\star$ are generated by discrete solutions of the relaxation $p'$.

\begin{proposition}
	\label{prop: recover_dirac_solution}
	If $\mu^\star = \sum_{i\in\I} \lambda_i \delta_{x_i}$ is a discrete probability measure that solves $p'$, then $x^\star \coloneqq x_i$ solves $p^\star$ for all $i\in\I$ such that $\lambda_i > 0$.
\end{proposition}

\begin{proof}
	Let $i^\star \in \argmin_{i\in\I} g(x_i)$. Since $\lambda_i \ge 0$ for all $i$, $\sum_{i\in\I} \lambda_i = 1$, and $g(x_{i^\star}) \le g(x_i)$ for all $i$, it holds that
	\begin{align*}
		p^\star &\le  g(x_{i^\star}) = \sum_{i\in\I} \lambda_i g(x_{i^\star}) \\
		&\le \sum_{i\in\I} \lambda_i g(x_i) = \E_{x'\sim \mu^\star} g(x') = p',
	\end{align*}
	so $x_{i^\star}$ solves $p^\star$ by Proposition~\ref{prop: exact_relaxation}. If $x_{i'}$ does not solve $p^\star$ for some $i'\in\I$ such that $\lambda_{i'} > 0$, then $p^\star = g(x_{i^\star}) < g(x_{i'})$, implying that $\sum_{i\in\I} \lambda_i g(x_{i^\star}) < \sum_{i\in\I} \lambda_i g(x_i)$ and hence that $p^\star < p'$, which contradicts Proposition~\ref{prop: exact_relaxation}.
\end{proof}

The above results show that we may solve the problem $p^\star$ of interest by solving $p'$ for a discrete optimal distribution. The remainder of this section is dedicated to this approach.

\subsection{Finite-Dimensional Reduction}
To make our finite-dimensional reduction, we recall the definitions of conjugate and perspective functions.

\begin{definition}
	\label{def: conjugate}
	The \emph{conjugate} of a function $f\colon \Rd \to \Rbar$ is the function $f^*\colon \Rd\to \Rbar$ defined by
	\begin{equation*}
		f^*(y) = \sup_{x\in\dom(f)} (y^\top x - f(x)).
	\end{equation*}
\end{definition}

We write $f^{**}$ to denote the biconjugate $(f^*)^*$.

\begin{definition}
	\label{def: perspective}
	The \emph{perspective} of a proper, closed, and convex function $f\colon \Rd \to \Rbar$ is the function $\prsp{f}\colon \Rd\times \Rpl \to \Rbar$ defined by
	\begin{equation*}
		\prsp{f}(x,t) = \begin{aligned}
			\begin{cases}
				t f(x / t)  & \text{if $t>0$}, \\
				\sup_{y\in \dom(f^*)} y^\top x & \text{if $t=0$}.
			\end{cases}
		\end{aligned}
	\end{equation*}
\end{definition}

Recall that the perspective $\prsp{f}$ of a convex function $f$ is also convex, and that the conjugate $f^*$ is convex even when $f$ is nonconvex \cite{boyd2004convex}.

Throughout the remainder of the paper, we fix $g$ and $X$ to be min-max affine and convex, respectively, via the following structural assumptions:

\begin{assumption}
	\label{ass: min-max_affine}
	It holds that $g\in\G$, taking the form $g(x) = \min_{i\in\I} g_i(x)$ with $g_i(x) = \max_{j\in\J}(a_{ij}^\top x + b_{ij})$.
\end{assumption}

\begin{assumption}
	\label{ass: support_set}
	The set $X$ takes the form $X = \{x\in\Rd : c_k(x) \le 0, ~ k\in\K\}$ with $c_k\colon\Rd\to\Rbar$ a proper, closed, and convex function for all $k\in\K$.
\end{assumption}

We now make the reduction by introducing two finite-dimensional convex optimization problems:

\begin{equation*}
	\begin{aligned}
		\underline{c} \coloneqq{} & \underset{\substack{\lambda_i,\eta_i\in\R \\ x_i \in\Rd}}{\text{minimize}} && \sum_{i\in\I} \eta_i \\
		& \text{subject to} && \prsp{c_k}(x_i,\lambda_i) \le 0, ~ i\in\I, ~ k\in\K, \\
		&&& \prsp{g_i}(x_i,\lambda_i) \le \eta_i, ~ i\in\I, \\
		&&& \sum_{i\in\I}\lambda_i = 1, ~ \lambda \ge 0.
	\end{aligned}
\end{equation*}
\begin{equation*}
	\begin{aligned}
		\overline{c} \coloneqq{} & \underset{\substack{\alpha,\beta_{ik}\in\R \\ y_i,z_{ik}\in\Rd}}{\text{maximize}} && -\alpha \\
		& \text{subject to} && g_i^*(y_i) + \sum_{k\in\K} \prsp{c_k^*}(z_{ik},\beta_{ik}) \le \alpha, ~ i\in \I, \\
		&&& y_i + \sum_{k\in\K} z_{ik} = 0, ~ i\in \I, \\
		&&& \beta_{ik} \ge 0, ~ i\in\I, ~ k\in\K.
	\end{aligned}
\end{equation*}

Intuitively, $\underline{c}$ is minimizing a sort of ``average'' of the components $g_i$ at a finite number of points $x_i$ with weights given by the probability vector $\lambda$, and $\overline{c}$ is its dual. We now leverage recent results in distributionally robust optimization to show that the finite reductions $\underline{c},\overline{c}$ allow us to solve the infinite-dimensional problem $p'$ under mild assumptions.

\begin{definition}
	\label{def: slater}
	Let $f_0,f_1,\dots,f_m$ and $h_1,\dots,h_n$ be extended real-valued functions defined on $\Rd$. The optimization problem $p = \inf\{f_0(x) : f_1(x)\le 0, \dots, f_m(x)\le 0, ~ h_1(x) = 0, \dots, h_n(x)=0, ~ x\in\Rd\}$ \emph{admits a Slater point} if there exists $x\in\bigcap_{i=0}^m \relint(\dom(f_i)) \cap \bigcap_{j=1}^n \relint(\dom(h_j))$ such that $f_i(x) \le 0$ and $h_j(x) = 0$ for all $i$ and all $j$, and such that $f_i(x) < 0$ for all $i\ne 0$ such that $f_i$ is nonlinear.
\end{definition}

\begin{assumption}
	\label{ass: slater}
	The set $X$ is bounded and the optimization problem $\overline{c}$ admits a Slater point.
\end{assumption}

The above boundedness assumption on $X$ is standard in the adversarial robustness literature. The Slater condition may be verified by simply solving $\overline{c}$ with a small number $\epsilon>0$ added to all of the nonlinear inequality constraints; replace $f_i(x) \le 0$ with $f_i(x)+\epsilon \le 0$ for all nonlinear constraint functions $f_i$.

\begin{theorem}
	\label{thm: certificate}
	If Assumption~\ref{ass: slater} holds, then $\underline{c} = p' = \overline{c}$, and the discrete probability distribution $\sum_{i\in\I : \lambda_i^\star \ne 0} \lambda_i^\star \delta_{x_i^\star / \lambda_i^\star}$ solves $p'$ for all solutions $(\eta^\star,\lambda^\star,x^\star)$ to $\underline{c}$.
\end{theorem}

\begin{proof}
	Since $X$ is defined by a finite intersection of $0$-sublevel sets of proper, closed, and convex functions (Assumption~\ref{ass: support_set}), and since every $g_i\colon x\mapsto \max_{j\in\J} (a_{ij}^\top x + b_{ij})$ is a proper, closed, and convex function, the result follows from \cite[Theorem~12(ii)]{zhen2021mathematical}.
\end{proof}

Theorem~\ref{thm: certificate} together with our Propositions~\ref{prop: exact_relaxation} and \ref{prop: recover_dirac_solution} show that we are able to \emph{exactly} compute an optimal attack solving the nonconvex problem $p^\star$ by solving the convex optimizations $\underline{c},\overline{c}$.

\subsection{Reformulating and Solving the Finite Reduction}

In order to solve $\underline{c},\overline{c}$, we must derive the appropriate conjugates and perspectives. In this subsection, we do so for the common cases where $X$ is defined in terms of norm balls or polyhedra. We will also see that computing the conjugate $g_i^*$ is highly nontrivial, and as a result we turn to tractably reformulating the constraint involving $g_i^*$ using duality theory.

\begin{proposition}
	\label{prop: max_affine_perspective}
	The perspective of $g_i\colon x\mapsto \max_{j\in\J} (a_{ij}^\top x + b_{ij})$ is given by $\prsp{g_i}(x,t) = \max_{j\in\J} (a_{ij}^\top x + b_{ij} t)$ for all $(x,t)\in\Rd\times\Rpl$.
\end{proposition}
\begin{proof}
	Let $x\in\Rd$. If $t>0$, then
	\begin{multline*}
		\prsp{g_i}(x,t) = t g_i(x / t) \\
		= t\max_{j\in\J} (a_{ij}^\top x / t + b_{ij}) = \max_{j\in\J} (a_{ij}^\top x + b_{ij}t).
	\end{multline*}
	If $t=0$, then
	\begin{align*}
		\prsp{g_i}(x,t) &= \liminf_{(x',t')\to(x,0)} \prsp{g_i}(x',t') \\
				&= \liminf_{(x',t')\to(x,0)} \max_{j\in\J} (a_{ij}^\top x' + b_{ij}t') \\
				&= \max_{j\in\J}a_{ij}^\top x = \max_{j\in\J} (a_{ij}^\top x + b_{ij} t),
	\end{align*}
	where the first equality comes from Theorem~13.3 and Corollary~8.5.2 in \cite{rockafellar1970convex} and the third equality comes from the continuity of $(x,t)\mapsto \max_{j\in\J} (a_{ij}^\top x + b_{ij}t)$.
\end{proof}

\begin{proposition}
	\label{prop: norm_perspective}
	The perspective of $c_k\colon x\mapsto \|x-\overline{x}\| - \epsilon$ is given by $\prsp{c_k}(x,t) = \|x-t\overline{x}\| - \epsilon t$ for all $(x,t)\in\Rd\times\Rpl$.
\end{proposition}
\begin{proof}
	Following the same reasoning as in the proof of Proposition~\ref{prop: max_affine_perspective}, we find that $\prsp{c_k}(x,t) = t(\|x / t - \overline{x}\| - \epsilon) = \|x-t \overline{x}\| - \epsilon t$ for $t>0$ and $\prsp{c_k}(x,t) = \liminf_{(x',t')\to(x,0)}(\|x'-t' \overline{x}\|-\epsilon t') = \|x\| = \|x-t \overline{x}\| - \epsilon t$ for $t=0$.
\end{proof}

\begin{proposition}
	\label{prop: norm_conjugate}
	The conjugate of $c_k\colon x\mapsto \|x-\overline{x}\|-\epsilon$ is given for all $z\in\Rd$ by
	\begin{equation*}
		c_k^*(z) = \begin{aligned}
		\begin{cases}
			z^\top \overline{x} + \epsilon & \text{if $\|z\|_* \le 1$}, \\
			\infty & \text{if $\|z\|_* > 1$}.
		\end{cases}
		\end{aligned}
	\end{equation*}
\end{proposition}
\begin{proof}
	Let $z\in\Rd$ be such that $\|z\|_* \le 1$. Then
	\begin{equation*}
	\sup_{x \in\Rd : x \ne \overline{x}} \frac{z^\top (x-\overline{x})}{\|x-\overline{x}\|} = \sup_{x'\in\Rd : \|x'\| \le 1} z^\top x' = \|z\|_* \le 1,
	\end{equation*}
	so $z^\top(x-\overline{x})-\|x-\overline{x}\| \le 0$ for all $x\ne \overline{x}$. Also, $z^\top(x-\overline{x})-\|x-\overline{x}\| = 0$ for $x=\overline{x}$, and therefore $\sup_{x\in\Rd} (z^\top (x-\overline{x}) - \|x-\overline{x}\|) = 0$, indicating that
	\begin{equation*}
		c_k^*(z) = \sup_{x\in\Rd} (z^\top(x-\overline{x})-\|x-\overline{x}\|) + z^\top \overline{x} + \epsilon = z^\top \overline{x}+\epsilon.
	\end{equation*}

	On the other hand, let $z\in\Rd$ be such that $\|z\|_*>1$. Then there exists $x'\in\Rd\setminus\{0\}$ such that $\frac{z^\top x'}{\|x'\|} >1$, implying that $z^\top x' - \|x'\| > 0$, and hence
	\begin{align*}
		c_k^*(z) &\ge z^\top (\overline{x}+\alpha x') - \|\alpha x'\| + \epsilon \\
			 &= \alpha(z^\top x' - \|x'\|) + z^\top \overline{x}+\epsilon \to\infty
	\end{align*}
	as $\alpha\to\infty$. Thus, $c_k^*(z) = \infty$.
\end{proof}

\begin{proposition}
	\label{prop: norm_conjugate_perspective}
	The perspective of the conjugate of $c_k\colon x\mapsto \|x-\overline{x}\| - \epsilon$ is given for all $(z,t)\in\Rd\times\Rpl$ by
	\begin{equation*}
		\prsp{c_k^*}(z,t) = \begin{aligned}
		\begin{cases}
			z^\top \overline{x} + \epsilon t & \text{if $\|z\|_* \le t$}, \\
			\infty & \text{if $\|z\|_* > t$}.
		\end{cases}
		\end{aligned}
	\end{equation*}
\end{proposition}

\begin{proof}
	Let $t>0$. If $z\in\Rd$ is such that $\|z\|_*\le t$, then $\|z / t\|_* \le 1$, so $\prsp{c_k^*}(z,t) = t c_k^*(z / t) = t((z / t)^\top \overline{x}+\epsilon) = z^\top \overline{x}+ \epsilon t$. If $\|z\|_* > t$, then $\|z / t\|_* > 1$, so $\prsp{c_k^*}(z,t) = t c_k^*(z / t) = \infty$.

	On the other hand, let $t = 0$. Then
	\begin{equation*}
		\prsp{c_k^*}(z,t) = \sup_{x\in \dom(c_k^{**})} z^\top x = \sup_{x\in\Rd} z^\top x = \begin{aligned}
		\begin{cases}
			0 & \text{if $z=0$}, \\
			\infty & \text{if $z\ne 0$},
		\end{cases}
		\end{aligned}
	\end{equation*}
	since $c_k^{**} = c_k$ which has domain $\Rd$, as $c_k$ is proper, closed, and convex \cite[Theorem~12.2]{rockafellar1970convex}. Since, when $t=0$, the condition $z=0$ is equivalent to $\|z\|_*\le t$ and the condition $z\ne 0$ is equivalent to $\|z\|_* > t$, the proof is complete.
\end{proof}

We also provide the conjugates and perspectives for polyhedral $X$:

\begin{proposition}
	\label{prop: polyhedral_X}
	Let $c_k\colon x \mapsto \psi_k^\top x + \omega_k$ for some $\psi_k\in\Rd$ and some $\omega_k\in\R$. Then the following all hold:
	\begin{enumerate}
		\item $\prsp{c_k}(x,t) = \psi_k^\top x + \omega_k t$,
		\item $c_k^*(z) = \begin{aligned}
		\begin{cases}
			-\omega_k & \text{if $z = \psi_k$}, \\
			\infty & \text{if $z\ne \psi_k$},
		\end{cases}
		\end{aligned}$
		\item and $\prsp{c_k^*}(z,t) = \begin{aligned}
		\begin{cases}
			-\omega_k t & \text{if $z = t \psi_k$}, \\
			\infty & \text{if $z\ne t \psi_k$}.
		\end{cases}
		\end{aligned}$
	\end{enumerate}
\end{proposition}
The proof of Proposition~\ref{prop: polyhedral_X} follows from a straightforward application of the definitions of conjugate and perspective, and is hence omitted for brevity.

The conjugate $g_i^*$ is all that remains to compute. However, although computing $g_i^*$ in closed form for the univariate ($d=1$) function $g_i\colon x\mapsto \max_{j\in\J} (a_{ij} x + b_{ij})$ can be straightforward, generalizing the formula to higher-dimensional settings is nontrivial. In theory, it is possible to express $g_i^*$ for $d>1$ in closed form via \cite[Theorem~19.2]{rockafellar1970convex}. However, this requires solving a vertex enumeration problem, i.e., determining finite sets $V,R\subseteq \Rd\times \R$ such that the polyhedron $\epi(g_i^*) \coloneqq \{(x,t)\in\Rd\times \R : \text{$a_{ij}^\top x + b_{ij} \le t$ for all $j\in\J$}\}$ equals $\conv(P)+\cone(R)$. The vertex enumeration problem is NP-hard in general \cite{khachiyan2009generating}. See the Minkowski-Weyl theorem \cite[Theorem~19.1]{rockafellar1970convex} for the theory on such representations of polyhedra. In Theorem~\ref{thm: max_affine_conjugate} that follows, we instead take a duality-based robust optimization approach to tractably deal with the conjugate $g_i^*$ in a direct manner.

\begin{lemma}
	\label{lem: domain_max_affine_conjugate}
	It holds that $\dom(g_i^*) = \conv\{a_{ij} : j\in\J\}$.
\end{lemma}
\begin{proof}
	Let $y\in\conv\{a_{ij} : j\in\J\}$. Then $y = \sum_{j\in\J}\theta_j a_{ij}$ for some $\theta\in\Rn$ such that $\theta \ge 0$ and $\sum_{j\in\J}\theta_j = 1$. Hence, for all $x\in\Rd$, we find that
	\begin{align*}
		y^\top x - \max_{j\in\J}(a_{ij}^\top x + b_{ij}) &= \sum_{j\in\J} \theta_j a_{ij}^\top x  - \max_{j\in\J}(a_{ij}^\top x + b_{ij}) \\
		&= \sum_{j\in\J}\theta_j(a_{ij}^\top x + b_{ij}) \\
		&\quad - \max_{j\in\J}(a_{ij}^\top x + b_{ij})  - \sum_{j\in\J}\theta_j b_{ij} \\
		&\le \sum_{j\in\J} \theta_j \max_{l\in\J} (a_{il}^\top x + b_{il}) \\
		&\quad - \max_{j\in\J}(a_{ij}^\top x + b_{ij}) - \sum_{j\in\J}\theta_j b_{ij} \\
		&= - \sum_{j\in\J}\theta_j b_{ij},
	\end{align*}
	and thus $g_i^*(y) \le -\sum_{j\in\J}\theta_j b_{ij} < \infty$, so $y\in\dom(g_i^*)$.

	On the other hand, let $y\in\dom(g_i^*)$, so that $g_i^*(y)<\infty$. An epigraphic reformulation of $g_i^*(y)$ yields that $\infty > g_i^*(y) = \sup_{x\in\Rd}(y^\top x - \max_{j\in\J}(a_{ij}^\top x + b_{ij})) = \sup_{(x,t)\in\Rd\times\R}\{y^\top x - t : \text{$a_{ij}^\top x + b_{ij}\le t$ for all $j\in\J$}\}$. This reformulation is a linear program with a finite optimal value, and hence by \cite[Proposition~3.1.3]{bertsekas2016nonlinear}, the reformulation is attained by some $(x,t)\in\Rd\times\R$, and since it must be the case that $t = a_{ij}^\top x + b_{ij}$ for some $j\in\J$ at this point $(x,t)$, we conclude that this $x$ solves the supremum defining $g_i^*(y)$ in its original form (i.e., pre-epigraphic reformulation). Therefore, by the first-order optimality condition for unconstrained convex optimization \cite[Theorem~23.2]{rockafellar1970convex}, it holds that $0\in\partial h_i(x)$, where $h_i\colon\Rd\to\R$ is the convex function defined by $h_i(x) = \max_{j\in\J}(a_{ij}^\top x + b_{ij}) - y^\top x$. Using the rules for subdifferentials of pointwise maxima and sums of proper convex functions \cite[Proposition~B.22]{bertsekas2016nonlinear},\cite[Theorem~23.8]{rockafellar1970convex}, we have that $\partial h_i(x) = \conv \left(\bigcup_{j\in \mathcal{A}(x)} \{a_{ij}\}\right) + \{-y\}$, where $\mathcal{A}(x)$ denotes the set of active indices at $x$: $\mathcal{A}(x) = \{j\in\J : a_{ij}^\top x + b_{ij} = \max_{l\in\J} (a_{il}^\top x + b_{il})\}$. Since $0\in\partial h_i(x)$, this yields that $y \in \conv \left(\bigcup_{j\in\mathcal{A}(x)}\{a_{ij}\}\right) \subseteq \conv\{a_{ij} : j\in\J\}$. This completes the proof.
\end{proof}

\begin{assumption}
	\label{ass: max_affine_nonredundant}
	The functions $g_i$ are nonredundant in the sense that for all $j\in\J$ there exists $x\in\Rd$ such that $g_i(x) = a_{ij}^\top x + b_{ij}$.
\end{assumption}

It is easy to see that nonredundancy of $g_i$ is efficiently verified by solving the linear (feasibility) programs $\inf\{0 : (a_{il}-a_{ij})^\top x + (b_{il} - b_{ij}) \le 0 ~ \text{for all $l\in\J$}, ~ x\in \Rd\}$ for all $j\in\J$. Removing the affine components of $g_i$ with infeasible programs ensures that Assumption~\ref{ass: max_affine_nonredundant} holds and does not change the model's predictions.

\begin{theorem}
	\label{thm: max_affine_conjugate}
	Suppose that Assumption~\ref{ass: max_affine_nonredundant} holds, and let $h\colon \Gamma \to \R$ be an arbitrary real-valued function defined on some nonempty set $\Gamma$. Then, for all $y\in\Rd$ and all $\gamma\in\Gamma$, it holds that $g_i^*(y) \le h(\gamma)$ if and only if, for all $j\in\J$, there exists $\nu_{ij}\in\Rn$ such that the following all hold:
	\begin{enumerate}
		\item $y = \sum_{j\in\J} \theta_j a_{ij}$ for some $\theta\in\Rn$ such that $\theta \ge 0$ and $\sum_{j\in\J}\theta_j = 1$,
		\item $\nu_{ij} \ge 0$,
		\item $y - a_{ij} + \sum_{l\in\J}(\nu_{ij})_l (a_{ij} - a_{il}) = 0$,
		\item and $-b_{ij}+\sum_{l\in\J} (\nu_{ij})_l (b_{ij} - b_{il}) \le h(\gamma)$.
	\end{enumerate}
\end{theorem}
\begin{proof}
	Let $y\in\Rd$ and $\gamma\in\Gamma$. If $y \ne \sum_{j\in\J}\theta_j a_{ij}$ for all $\theta\in\Rn$ such that $\theta \ge 0$ and $\sum_{j\in\J}\theta_j = 1$, then $y \notin \conv\{a_{ij} : j\in\J\}$ and hence $y \notin \dom(g_i^*)$ by Lemma~\ref{lem: domain_max_affine_conjugate}. In this case, $g_i^*(y) = \infty > h(\gamma)$ since $h$ is real-valued. Therefore, the first condition enumerated in the theorem is necessary for $g_i^*(y) \le h(\gamma)$.

	Going forward, assume that $y = \sum_{j\in\J}\theta_j a_{ij}$ for some $\theta\in\Rn$ such that $\theta \ge 0$ and $\sum_{j\in\J}\theta_j = 1$. Hence, $g_i^*(y)<\infty$. Breaking up the conjugate's supremum into $n$ suprema over the affine components of $g_i$ yields
	\begin{align*}
		g_i^*(y) &= \sup_{x\in\Rd} (y^\top x - \max_{j\in\J}(a_{ij}^\top x + b_{ij})) \\
			 &= \max_{j\in\J} \sup_{x\in\Rd}\{(y-a_{ij})^\top x - b_{ij} : \\
			 &\quad \text{$(a_{il}-a_{ij})^\top x + (b_{il}-b_{ij}) \le 0$ for all $l\in\J$}\}.
	\end{align*}
	Denote the inner suprema by $\mathfrak{p}_{ij} \coloneqq \sup_{x\in\Rd}\{(y-a_{ij})^\top x - b_{ij} : \text{$(a_{il}-a_{ij})^\top x + (b_{il}-b_{ij}) \le 0$ for all $l\in\J$}\}$. Since, by Assumption~\ref{ass: max_affine_nonredundant}, for all $j\in\J$ there exists $x\in\Rd$ such that $\max_{l\in\J}(a_{il}^\top x + b_{il}) = g_i(x) = a_{ij}^\top x + b_{ij}$, it holds that $\{x\in\Rd : \text{$a_{ij}^\top x + b_{ij} \ge a_{il}^\top x + b_{il}$ for all $l\in\J$}\} \ne \emptyset$ for all $j\in\J$, implying that every $\mathfrak{p}_{ij}$ is feasible, i.e., $\mathfrak{p}_{ij}>-\infty$. Furthermore, since $g_i^*(y) < \infty$, it must be the case that $\mathfrak{p}_{ij}<\infty$ for all $j\in\J$. Thus, every optimal value $\mathfrak{p}_{ij}$ is finite. Therefore, by \cite[Proposition~3.1.3]{bertsekas2016nonlinear}, every $\mathfrak{p}_{ij}$ is attained, and therefore by \cite[Proposition~4.4.2]{bertsekas2016nonlinear} strong duality holds between $\mathfrak{p}_{ij}$ and its dual problem, which we denote by $\mathfrak{d}_{ij}$, and it also holds that $\mathfrak{d}_{ij}$ is attained. A routine derivation via Lagrangian duality therefore yields that
	\begin{align*}
		\mathfrak{p}_{ij} &= \mathfrak{d}_{ij} \\
		&= \inf_{\nu_{ij}\in\Rn}\bigg\{\sum_{l\in\J}(\nu_{ij})_l (b_{ij} - b_{il}) - b_{ij} : \\
				  & \quad y-a_{ij}+\sum_{l\in\J} (\nu_{ij})_l (a_{ij} - a_{il}) = 0, ~ \nu_{ij} \ge 0\bigg\}.
	\end{align*}
	Hence, $g_i^*(y) \le h(\gamma)$ if and only if $\max_{j\in\J} \mathfrak{p}_{ij} \le h(\gamma)$ if and only if $\mathfrak{p}_{ij} \le h(\gamma)$ for all $j\in\J$. Thus, since $\mathfrak{d}_{ij}$ is attained, it holds that $g_i^*(y) \le h(\gamma)$ if and only if, for all $j\in\J$, there exists $\nu_{ij}\in\Rn$ such that $\nu_{ij} \ge 0$, $y-a_{ij}+\sum_{l\in\J} (\nu_{ij})_l (a_{ij} - a_{il}) = 0$, and $-b_{ij}+\sum_{l\in\J} (\nu_{ij})_l (b_{ij} - b_{il}) \le h(\gamma)$. This completes the proof.
\end{proof}

With the above conjugate and perspective derivations, our reformulations of $\underline{c},\overline{c}$ are complete; they may now be directly solved using off-the-shelf convex optimization solvers.

\begin{remark}
	Our developments can be generalized, so long as one can compute the appropriate conjugates and perspectives. In particular, the mathematical machinery yielding a discrete distribution solution to $p'$ from a solution to an associated finite-dimensional convex optimization problem may be applied to general convex functions $g_i$ and other (non-norm-based and non-polyhedral) convex attack sets $X$ \cite{zhen2021mathematical}. In fact, moment constraints on $\mu\in\P(X)$ may even be added to the semi-infinite program $p'$, which may allow for modeling alternative ``distributional attacks'' beyond the standard ``Dirac attack'' at a single point considered here.
\end{remark}

\section{Experiments}
\label{sec: exper}

In this section, we illustrate the utility of our method in both robust control and image classification settings.\footnote{All experiments are conducted on a Ubuntu 22.04 instance with an Intel i7-9700K CPU and NVIDIA RTX A6000 GPU.}

\subsection{Robust Control Certification}
\label{sec: controller}

We take an illustrative robust control example adapted from the well-known autonomous vehicle collision avoidance problem \cite{ren2022chance,wang2020non}. Consider two planar vehicles approaching an intersection located at the origin $(0,0)\in\R^2$. One vehicle travels east with state $\bx(t) = \left(x(t), \dot{x}(t)\right)\in\R^2$ at time $t$. The other vehicle, which we control and hence term the ``ego vehicle,'' travels north with state $\by(t) = \left(y(t), \dot{y}(t)\right)$. The eastbound uncontrolled vehicle has a fixed velocity ($\ddot{x}(t) = 0$ for all $t$). The full state $\left(x(t), \dot{x}(t), y(t), \dot{y}(t)\right)$ is randomly initialized at $t=0$ within $[-3, -2] \times [1 / 2, 5 / 2] \times [-3, -2] \times [0, 2]$. The vehicles are each $1$ unit long and $1/2$ unit wide, matching the width of the road. Thus, a vehicle is considered to be in the intersection if the absolute value of its position is less than $3/4$. If the vehicles collide, the simulation is stopped. We simulate standard double integrator dynamics with a time step $\Delta t = 0.05$ for $100$ steps.

We control the northbound vehicle using a learned policy $u(t) = -\poll\left(\bx(t), \by(t)\right)$ that enters the dynamics as $\ddot{y}(t) = \Pi_{[-1,1]} \left(u(t)\right)$, where $\poll \colon \R^4\to\R$ is a min-max affine function with $m=n=10$ and $\Pi_{[-1,1]}$ is the natural projection mapping of $\R$ onto $[-1,1]$. Our robustness certificates apply for all training schemes, e.g., reinforcement learning and imitation learning. We train $\poll$ using imitation learning on $500$ trajectories generated by a hand-programmed expert policy $\pole$. We use the mean squared error loss function and train for $20$ epochs using the Adam optimizer at a learning rate of $0.01$. The expert policy $\pole$ is designed to stop the ego vehicle $\delta=0.1$ units before the intersection with a constant acceleration, then apply no acceleration until the tail of the uncontrolled vehicle is $\delta$ units past the intersection, and then accelerate with the maximum input of $1$.

We now consider certifying the safety of our control system. Our goal is to guarantee that the ego vehicle always brakes when the uncontrolled vehicle is approaching or inside the intersection. This enforcement of braking corresponds to ensuring that the largest acceleration signal $u(t)$ is less than zero, which amounts to minimizing the output of $\poll$ over the set of states for which we desire braking. This is formalized by requiring braking for all states in the set
\begin{equation*}
	X = [-3+\delta,\tfrac{3}{4}]\times [\tfrac{1}{2}+\delta,\tfrac{5}{2}-\delta] \times [-3+\delta, -\tfrac{3}{4}] \times [\delta,2-\delta],
\end{equation*}
which consists of states where the uncontrolled vehicle is approaching or in the intersection and the ego vehicle is approaching the intersection. The small positive constant $\delta = 0.1$ accounts for boundary states where expert trajectories may not have been sampled.

Utilizing our robustness certificates from Section~\ref{sec: theory}, we verify that indeed $u(t) = -\poll\left(\bx(t),\by(t)\right) < 0$ for all states $(\bx(t),\by(t))\in X$. For visual purposes, we also consider fixing a particular $\by(t)$ and computing the largest possible acceleration $u(t)$ amongst all uncontrolled vehicle states $\bx(t)$ captured by $X$. The solutions to this problem over a range of $\by(t)$ are plotted in Figure~\ref{fig: controls}. As expected, for all $\by(t)$, the ego vehicle is braking. As the ego vehicle approaches the intersection (large $y(t)$) or becomes faster (large $\dot{y}(t)$), we certify that the controller brakes more heavily.

\begin{figure}
    \centering
    \input{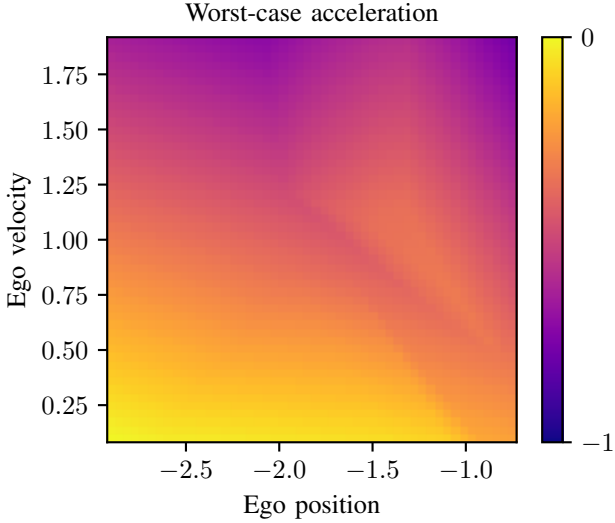}
    \vspace*{-1.5\baselineskip}
    \caption{Largest possible acceleration over all uncontrolled vehicle states for particular values of the ego vehicle state. The output is always negative, ensuring some level of braking.}
    \label{fig: controls}
\end{figure}

\subsection{Image Classification}
\label{sec: image}

We demonstrate the tightness and efficiency of our method on an image classification example adapted from \cite{pfrommer2023asymmetric}. The task is to distinguish between two visually similar MNIST classes: the digits $3$ and $8$ \cite{lecun1998mnist}. As we consider the asymmetric setting, we aim to certify predictions for one particular class, which we take to be the class of $3$'s, while maintaining high clean accuracy for both classes. We consider the attack set $X = \{x\in\Rd : \|x-\overline{x}\|_\infty \le \epsilon\}$ over a range of radii $\epsilon>0$ around test images $\overline{x}$. In this setting, certificates ensure that pixelwise adversarial alterations of an image $\overline{x}$ of a $3$ cannot fool the classifier into predicting an $8$.

We compare two approaches: 1) directly learning our min-max representation with $n=m=15$ and certifying via our convex optimization-based certificates, and 2) learning a standard composition-based ReLU model and certifying via the state-of-the-art verifier $\alpha,\beta$-CROWN \cite{wang2021beta}. Since $\alpha,\beta$-CROWN's worst-case runtime scales exponentially with model size, we instantiate the standard ReLU model with one hidden layer and $100$ hidden units, which is the smallest hidden layer size that yields comparable clean accuracy to our min-max representation. We use adversarial training (see \cite{madry2018towards}) with $\ell_\infty$-attacks starting at a radius of $0.001$ and linearly interpolate to a radius of $\epsilon_\textup{train}$ over the first $20$ epochs, where $\epsilon_\textup{train}=0.05$ for our model and $\epsilon_\textup{train}=0.3$ for the standard ReLU model. Both models are trained using the Adam optimizer with a learning rate of $0.001$ for $60$ epochs.

Figure~\ref{fig: mnist} compares the certified accuracy (averaged over the test inputs) of our method against that of $\alpha,\beta$-CROWN. As certifying at a particular $\epsilon$ using our method is fast, for each test input, the largest certifiable $\ell_\infty$-radius is found using binary search in order to yield a smooth certified accuracy curve. On the other hand, due to the expensive runtime of $\alpha,\beta$-CROWN, we only certify at the select radii shown. Our min-max representation exceeds the state-of-the-art baseline certified radii at far faster runtimes: certifying a single input-radius pair $(\overline{x},\epsilon)$ takes on average $3.67$ seconds with $\alpha,\beta$-CROWN versus only $0.48$ seconds with our method. We note that our runtime comparisons are solely based off of models with equivalent clean accuracy. Due to space constraints, we leave more thorough analyses of relative expressivity and computational complexity for future work.

\begin{figure}
    \centering
    \input{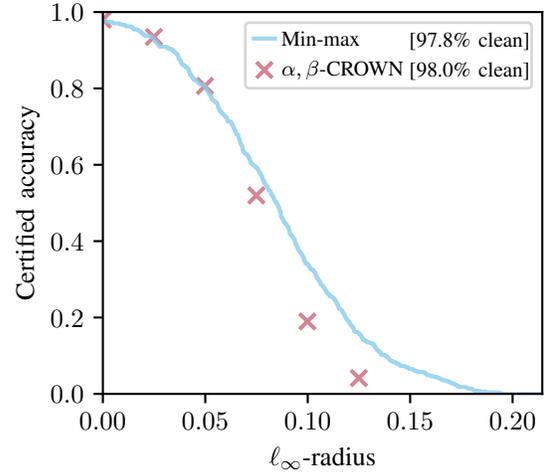}
    \caption{Certified accuracies of our min-max representation and of $\alpha,\beta$-CROWN on the MNIST $3$-versus-$8$ dataset.}
    \label{fig: mnist}
\end{figure}

\section{Conclusions}
\label{sec: conc}

In this work, we \emph{exactly} solve the nonconvex robustness certification problem over convex attack sets for min-max representations of ReLU neural networks by developing a tractable convex reformulation. An interesting line of future work may include developing more efficient min-max representations or estimations for arbitrary ReLU neural networks, so that the advantageous optimization properties derived in this paper may be easily applied. Other interest lies in comparing the number of affine regions of a general min-max affine function versus that of a general ReLU neural network.


\bibliographystyle{IEEEtran}
\bibliography{references}

\end{document}